\documentclass[reqno,b5paper]{amsart}
\usepackage{amssymb}
\usepackage{dsfont}
\usepackage{color}

\newtheorem{theorem}{Theorem}[section]

\newtheorem{proposition}{Proposition}[section]
\newtheorem{corollary}{Corollary}[section]
\newtheorem{lemma}{Lemma}[section]
\newtheorem{example}{Example}[section]
\newtheorem{remark}{Remark}[section]
\newtheorem{definition}{Definition}[section]

\newtheorem{claim}{Claim}[section]

\newcommand{\T}{\mathbb{T}}
\newcommand{\Z}{\mathbb{Z}}
\newcommand{\N}{\mathbb{N}}

\catcode`\@=12
\def\T{{\mathbb T}}
\def\eps{{\varepsilon}}

\def\Z{{\mathbb Z}}
\def\N{{\mathbb N}}
\def\R{{\mathbb R}}
\def\Q{{\mathbb Q}}

\makeindex

\begin{document}

\title[Topological torsion elements via natural density]{Topological torsion elements via natural density and a quest for solution of Armacost like problem}

\author{Pratulananda Das}

\address{Department of Mathematics, Jadavpur University, Kolkata-700032, India}
\email {pratulananda@yahoo.co.in}

\author{Ayan Ghosh}

\address{Department of Mathematics, Jadavpur University, Kolkata-700032, India}
\email {ayanghosh.jumath@gmail.com}

\begin{abstract}

One can use the number theoretic idea of the notion of natural density \cite{B1} to define topological s-torsion elements (which form the statistically characterized subgroups, recently developed in \cite{DPK}) extending Armacost's idea of topological torsion elements. We follow in the line of Armacost who had posed the famous classical problem for "description of topological torsion elements" of the circle group. In this note we consider the natural density version of Armacost's problem and present a complete description of topological s-torsion elements   in terms of the support, for all arithmetic sequences which also provides the solution of Problem 6.10 posed in \cite{DPK} .
\end{abstract}
\maketitle
%
%

\section{Introduction and the background}

Throughout $\R$, $\Q$, $\Z$ and $\N$ will stand for the set of all real numbers, the set of all rational numbers,
the set of all integers and the set of all natural numbers respectively. The first three are equipped with their usual abelian group structure,
The circle group $\T$ is identified with the quotient group $\R/\Z$ of $\R$ endowed with its usual compact topology. For $x\in\R$ we denote by $\{x\}$ the distance from the integers i.e. the difference $x - [x]$.

Recall that an element $x$ of an abelian group $X$ is torsion  if there exists $k \in \mathbb{N}$ such that $kx = 0$ (more specifically called $k$-torsion in this case). An element $x$ of an abelian topological group $G$ is  \cite{B}:
\begin{itemize}
\item[(i)] {\em topologically torsion} if $n!x \rightarrow 0;$
\item[(ii)] {\em topologically $p$-torsion}, for a prime $p$, if $p^nx \rightarrow 0.$
\end{itemize}

It is obvious that any $p$-torsion element is topologically $p$-torsion. Armacost \cite{A} defined the subgroups
$$ X_p = \{x \in X: p^nx \rightarrow 0\} \ ~\mbox{and} ~ \ X! =  \{x \in X: n!x \rightarrow 0\}$$
of an abelian topological group $X$,  and started their investigation, in particular description of the elements forming the subgroups. Note that the above two notions are just special cases of the following general notion considered in (Section 4.4.2, \cite{DPS}).

\begin{definition}
Let $(a_n)$ be a sequence of integers. An element $x$ in an abelian topological group $G$ is called {\em (topological) $\underline{a}$-torsion element} if $a_nx \rightarrow 0$.
\end{definition}

Further
$$
t_{(a_n)}(\T) := \{x\in \T: a_nx \to 0\mbox{ in } \T\}.
$$
of $\T$ is called {\em a characterized} $($by $(a_n))$ {\em subgroup} of $\T$, the term {\em characterized} appearing much later, coined in \cite{BDS}.

(a) Let $p$ be a prime. For the sequence $(a_n)$, defined by $a_n = p^n$ for every $n$, obviously $t_{(p^n)}(\T)$ contains
the Pr\" ufer group $\Z(p^\infty)$. Armacost \cite{A} proved that $t_{(p^n)}(\T)$ simply coincides with $\Z(p^\infty)$ and $x$ is a topologically $p$-torsion element iff. $supp(x)$ (defined later) is finite.

(b) Armacost \cite{A} posed the problem to describe the group $\T! = t_{(n!)}(\T)$.
It was resolved independently and almost simultaneously in \cite[Chap. 4]{DPS}
and by J.-P. Borel \cite{Bo2}.

In particular, in both the above mentioned instances, the sequences of integers, concerned are arithmetic sequences. Recall that
 a sequence of positive integers $(a_n)$ is called an arithmetic sequence if
 $$
 1<a_1<a_2<a_3< \ldots <a_n< \ldots \ ~ \ \mbox{ and } a_n|a_{n+1} \ ~ \ \mbox{ for every } n\in\N  .
 $$
As has been seen, some of the most interesting cases studied are the topological $\underline{a}$-torsion elements characterized by arithmetic sequences. The results of Armacost \cite{A} and Borel \cite{Bo2} were considered in full general settings with arbitrary arithmetic sequences in \cite{DD} and then with more clarity in \cite{DI1} where topological $\underline{a}$-torsion elements were completely described for a major class of arithmetic sequences.

However in certain sense, $a_nx \to 0 $ seems somewhat too restrictive, for example observe that, only countably many elements of $\T$ are topologically $p$-torsion even if the sequence $(a_n) = (p^n)$ is not too dense (it is a geometric progression, so has exponential growth). Motivated by this observation, we intend to consider a modified definition using a more general notion of convergence, as follows:

For $m,n\in\mathbb{N}$ and $m\leq n$, let $[m, n]$ denotes the set $\{m, m+1, m+2,...,n\}$. By
 $|A|$ we denote the cardinality of a set $A$. The lower and the upper natural densities of $A \subset \mathbb{N}$ are defined by \cite{B1}
$$
\underline{d}(A)=\displaystyle{\liminf_{n\to\infty}}\frac{|A\cap [0,n-1]|}{n} ~~\mbox{and}~~
\overline{d}(A)=\displaystyle{\limsup_{n\to\infty}}\frac{|A\cap [0,n-1]|}{n}.
$$
If $\underline{d}(A)=\overline{d}(A)$, we say that the natural
density of $A$ exists and it is denoted by $d(A)$ \cite{B1} (see \cite{S1, SZ} for related works). For two subsets $A, B$ of $\N$, we will write $A \subseteq^d B$ if $d(A\setminus B) = 0$ and $A =^d B$ if $d(A\triangle B) = 0$. The idea of natural density was later used to define the notion of statistical convergence (\cite{F, St}, see also \cite{C, S,Fr}).

\begin{definition}\label{Def1}
A sequence of real numbers $(x_n)$ is said to converge to a real number $x_0$ statistically if for any $\eps > 0$, $d(\{n \in \mathbb{N}: |x_n - x_0| \geq \eps\}) = 0$.
\end{definition}

Over the years, the notion of statistical convergence has been extended to general topological spaces using open neighborhoods \cite{MK} and a lot of work has been done on the notion of statistical convergence primarily because it extends the notion of usual convergence very naturally preserving many of the basic properties but at the same time including more sequences under its purview paving way for interesting applications (for example see \cite{BDK, C, C2}). The following two characterizations of statistical convergence would be of much help in our investigations.
\begin{theorem}
\cite{S} For a sequence of real numbers $(x_n)$, $x_n \to x_0$ statistically if and only if there exists a subset A of $ \N$ of asymptotic density 1, such that $\displaystyle{\lim_{n \in A}} x_n = x_0$.
\end{theorem}
\begin{lemma}\label{suf1}
(Folklore) A sequence $(x_n)$ converges to $\xi\in\R$ {\em statistically} if and only if for any $A\subseteq\N$ with $\overline{d}(A)>0$, there exists an infinite $A'\subseteq A $ such that $\lim\limits_{n\in A'} x_n = \xi $.
\end{lemma}

 Under the circumstances, it seems very natural to use the notion of statistical convergence to relax the condition $a_nx \to 0 $, i.e. we will consider the situation when $a_nx \to 0$ statistically (rephrasing Definition 1.2, in our context, this means that for every $\eps > 0$ there exists a subset A of $ \N$ of asymptotic density 0, such that $\{a_nx\} < \eps$ for every $n \not \in A$).

 \begin{definition}
Let $(a_n)$ be a sequence of integers. An element $x$ in an abelian topological group $G$ is called {\em topological $s_{\underline{a}}$-torsion element (only topological $s$-torsion when the given sequence $(a_n)$ is fixed)} if $a_nx \to 0\  \mbox{ statistically in }$ $G$.
\end{definition}
It has already been observed in the recent article \cite{DPK} that the topological $s$-torsion elements of $\T$ form a proper Borel subgroup of $\T$ of cardinality $\mathfrak c$ for any arithmetic sequence $(a_n)$ which was named as an {\em $s$-characterized (by $(a_n)$) subgroup of $\T$} denoted by $t_{(a_n)}^s (\T)$ i.e. $t^s_{(a_n)}(\T) := \{x\in \T: a_nx \to 0\  \mbox{ statistically in }\  \T\}$.

In this note we intend to present a complete description of topological $s$-torsion elements of $\T$ (given in Theorem 2.1 and Theorem 2.2) which in turn would answer the following open problem posed in \cite{DPK}:\\
\textbf{Problem 6.10.} Let $(a_n)$ be an arithmetic sequence such that $q_n>2$ for infinitely  many $n$. Does there
exist a characterization of the elements of the subgroup $t^s_{(a_n)}(\T)$ only in terms of the support.

Before proceeding to our main result we present below certain basic definitions, notations and results which will be needed in the next section.

\begin{definition}
 An arithmetic sequence of integers $(a_n)$ is called
 \begin{itemize}
 \item[(i)] $q$-bounded if the sequence of ratios $(q_n)$ (where $q_n=\frac{a_n}{a_{n-1}}$) is bounded.
 \item[(ii)] $q$-divergent if the sequence of ratios $(q_n)$  diverges to $\infty$.
 \end{itemize}
\end{definition}

\begin{lemma}\cite{DI1}
For any arithmetic sequence $(a_n)$ and $x\in\T$, we can build a unique sequence of integers $(c_n)$, where $0\leq c_n<q_n$, such that
\begin{equation}\label{canrep}
   x=\sum\limits_{n=1}^{\infty}\frac{c_n}{a_n}
 \end{equation}
    and $c_n<q_n-1$ for infinitely many $n$.
\end{lemma}

For $x\in\T$ with canonical representation (\ref{canrep}), we define $supp(x)=\{n\in\N\ : \ c_n\neq0\}$ and $supp_q(x)=\{n\in\N\ : \ c_n=q_n-1\}$. Clearly $supp_q(x) \subseteq supp(x)$.

Now we are going to introduce some equations which will be used repeatedly in our next section. Let $(a_n)$ be an arithmetic sequence and $x\in\T$ has canonical representation (\ref{canrep}). Then we have, for all non-negative integer $k$,
\begin{equation}\label{eq1}
i) \ ~ \ q_n\cdot q_{n+1}\ldots q_{n+k} \ = \ \frac{a_n}{a_{n-1}} \cdot \frac{a_{n+1}}{a_n} \ldots \frac{a_{n+k}}{a_{n+k-1}} \ = \ \frac{a_{n+k}}{a_{n-1}}.
\end{equation}
$$
ii) \ ~ \ \{a_{n-1}x\}\ = \ \sum\limits_{i=n}^{\infty} \frac{c_i}{a_i} \cdot a_{n-1}  \ = \ (c_n\cdot\frac{a_{n-1}}{a_n} + c_{n+1}\cdot\frac{a_{n-1}}{a_{n+1}}+ \ldots)$$
\begin{equation}\label{eq2}
 = \ \frac{c_n}{q_n}+\frac{c_{n+1}}{q_n q_{n+1}} + \ldots + \frac{c_{n+k}}{q_n \cdot q_{n+1} \ldots q_{n+k}} + \sum\limits_{i=k+1}^{\infty} \frac{c_{n+i}}{q_n \cdot q_{n+1} \ldots q_{n+i}}.
\end{equation}

$$
iii) \ ~ \ \{a_{n+k}x\} \ = \ a_{n+k}(\frac{c_{n+k+1}}{a_{n+k+1}}+\frac{c_{n+k+2}}{a_{n+k+2}}+ \ldots) \ = \ (\frac{c_{n+k+1}}{q_{n+k+1}}+\frac{c_{n+k+2}}{q_{n+k+1}\cdot q_{n+k+2}}+\ldots)
$$
\begin{equation}\label{eq3}
= \ q_n \cdot q_{n+1} \ldots q_{n+k} \sum\limits_{i=k+1}^{\infty} \frac{c_{n+i}}{q_n \cdot q_{n+1} \ldots q_{n+i}}.
\end{equation}
From equations (\ref{eq3}) and equations (\ref{eq2}), we get
\begin{equation}\label{eq4}
iv) \ ~ \ \{a_{n-1}x\} \ = \ \frac{c_n}{q_n} + \frac{c_{n+1}}{q_n\cdot q_{n+1}} + \ldots + \frac{c_{n+k}}{q_n\cdot q_{n+1} \ldots q_{n+k}}+ \frac{\{a_{n+k}x\}}{q_n\cdot q_{n+1} \ldots q_{n+k}}.
\end{equation}
For all $n\in\N$ and $k\in\N\cup \{0\}$, we define
\begin{equation}\label{eq0}
\sigma_{n,k}\ =\ \frac{c_n}{q_n} + \frac{c_{n+1}}{q_n\cdot q_{n+1}} + \ldots + \frac{c_{n+k}}{q_n\cdot q_{n+1} \ldots q_{n+k}}.
\end{equation}
Therefore, equation (\ref{eq4}) becomes,
\begin{equation}\label{eq5}
v) \ ~ \ \{a_{n-1}x\}= \sigma_{n,k}+\frac{\{a_{n+k}x\}}{q_n\cdot q_{n+1} \ldots q_{n+k}}.
\end{equation}
Further putting $k=0$ in equation (\ref{eq4}), we finally obtain
\begin{equation}\label{eq6}
vi) \ ~ \ \{a_{n-1}x\}=\frac{c_n}{q_n}+\frac{\{a_nx\}}{q_n}.\\
\end{equation}

Let $a = (a_n)$ be a given arithmetic sequence. Now for any $B\subseteq\N$ with $d(B)>0$, let $t_{(a_B)} (\T) = \{ x\in\T : \ \lim\limits_{n\in B} a_nx=0$ in $\T \}$ and $t_{(a_B)}^s (\T) = \{ x\in\T : \ \lim\limits_{n\in B'} a_nx=0$ in $\T\ $ for some $B'\subseteq B$ with $d(B\setminus B')=0\}$. Therefore, for all $B \subseteq \N$ with $\overline{d}(B)>0$, we have  $t_{(a_n)}^s (\T) \subseteq t_{(a_B)}^s (\T)$ and $t_{(a_n)}^s (\T) = \bigcap\limits_{B\in [\N]^{\aleph_0} \ \& \ \overline{d}(B)>0} t_{(a_B)}^s (\T)$.
\begin{lemma}\label{nec}
If $B\subseteq\N$ with $\overline{d}(B)>0$ and $x\in t_{(a_{B-1})}(\T)$ (where $B-1 = \{k-1: k \in B\}$) then the following hold: \\
  i) If $B\subseteq^d supp(x)$ and $q$-bounded, then $B\subseteq^d supp_q(x)$ and there exists $B'\subseteq B$ with $d(B\setminus B')=0$ such that $\lim\limits_{n\in B'} \{a_{n-1}x\}=1$ in $\R$.\\
  ii) If $d(B\cap supp(x))=0$, then there exists $B'\subseteq B$ with $d(B\setminus B')=0$ such that $\lim\limits_{n\in B'} \{a_{n-1}x\}=0$ in $\R$.
\end{lemma}
\begin{proof}
i) Let $q=1+\max\limits_{n\in B} \{q_n\}$ and $B'=B\cap supp(x)$. Since $B'\subseteq B$ and $d(B\setminus supp(x))=0$, we get $d(B\setminus B')=d(B\setminus (B\cap supp(x)))=d(B\setminus supp(x))=0$.\\
Therefore,
$$
\{a_{n-1}x\}\geq \frac{c_n}{q_n} > \frac{1}{q} \ ~ \ \forall \ n\in B' \mbox{ (Since $c_n\geq 1$ for all $n\in B'$) .}
$$
But as $x\in t_{(a_{B-1})}(\T)$, thus we can conclude that $ \lim\limits_{n\in B'} \{a_{n-1}x\}=1$ in $\R$. \\
Therefore,
\begin{eqnarray*}
1-\frac{1}{q_n} &<& 1- \frac{1}{q} <\{a_{n-1}x\} = \frac{c_n}{q_n} + \frac{\{a_nx\}}{q_n}\\
 &<& \frac{c_n+1}{q_n} \ ~ \ \mbox{for almost all} \ n\in B' \mbox{ (From equation (\ref{eq6})).}
\end{eqnarray*}
$$
\Rightarrow q_n -1<c_n+1 \mbox{ i.e. } c_n>q_n-2 \ \mbox{for almost all }\ n\in B'.
$$
Hence, $c_n=q_n-1$ for almost all $n\in B'$ i.e. $B'\subseteq^* supp_q(x)$, which implies $B\subseteq^d supp_q(x)$.\\

ii) Let $B'=B\cap(\N\setminus supp(x))$. Since $B'\subseteq B$ and $d(B\cap supp(x))=0$, we get $d(B\setminus B')=d(B\setminus (B\cap (\N\setminus supp(x))))=d(B\setminus (B\setminus supp(x)))=d(B\cap supp(x))=0$.
Now, from equation (\ref{eq6}), we have
$$
\{a_{n-1}x\}=0+\frac{\{a_nx\}}{q_n} <\frac{1}{2} \ \forall\ n\in B' \mbox{ (Since $c_n=0 \ \forall \ n\in B'$). }
$$
Then in view of the fact that $x\in t_{(a_{B-1})}(\T)$, we must have, $\lim\limits_{n\in B'} \{a_{n-1}x\}=0$ in $\R$.
\end{proof}

\begin{lemma}\label{suf2}
Let $(a_n)$ be an arithmetic sequence. Consider $A\subseteq\N$ with $\overline{d}(A)>0$ where $A$ is not $q$-bounded. If $\nexists$ any $q$-bounded subset $A'\subseteq A$ with $\overline{d}(A')>0$  then $\exists$ a $q$-divergent set $B\subseteq A$ with $d(A\setminus B)=0$.
\end{lemma}
\begin{proof}
Let $A_m=\{n\in A\ :  \ q_n=m\}$ for some $m\in \N\setminus \{1\}$. Since $\nexists$ any $q$-bounded set $A'\subseteq A$ with $\overline{d}(A')>0$, we conclude that $d(A_m)=0$ for all $m\in\N\setminus \{1\}$. If there exists $k\in\N$ such that $A_m$ is finite for all $m>k$, then setting $B=\bigcup\limits_{m=k+1}^{\infty} A_m$, it is easy to observe that $B$ is $q$-divergent with $d(A\setminus B)=0$ (since $d(\bigcup\limits_{m=2}^k A_m)=0$).

Otherwise without any loss of generality, we can assume that $A_m$ is infinite for all $m\in\N\setminus \{1\}$. Now, considering the sequence $(A_m)_{m\in\N}$ of density zero sets, one can find $C \subset \N$ with $d(C)=0$ such that $A_m\setminus C$ is finite for all $m\in\N\setminus\{1\}$ (for explicit construction of such a set, see \cite{S}, also the existence follows from the fact that $\mathcal{I}_d =\{A \subset \N: d(A) = 0\}$ is a $P$-ideal, see \cite{C}).

 Let, $B=A\setminus C$. Since $d(C)=0$, we conclude that $d(A\setminus B)=0$. Consider any $l\in\N\setminus \{1\}$. Let, $n_l= \ max\ \{n:\ n\in A_m\setminus C$ and $m\leq l\}$. Now, for all $n\in B$ with $n>n_l$, we have $q_n>l$. Since $l$ was taken arbitrarily, it follows that $B$ is $q$-divergent.
\end{proof}

\section{Main results}

\begin{theorem}\label{mainth}
Let, $(a_n)$ be an arithmetic sequence and $x\in\T$. Then $x$ is a topological $s$-torsion element (i.e. $x\in t_{(a_n)}^s (\T)$) if and only if either $d(supp(x))=0$ or if $\ \overline{d}(supp(x))>0$, then for all $A\subseteq\N$ with $\overline{d}(A)>0$ the following holds:
\begin{itemize}
\item[(a)]  If $A$ is $q$-bounded, then:
\begin{itemize}
\item[(a1)] If $A\subseteq^d supp(x)$, then $A+1\subseteq^d supp(x), \ A\subseteq^d supp_q(x)$ and there exists $A'\subseteq A$ with $d(A\setminus A')=0$ such that $\lim\limits_{n\in A'} \frac{c_{n+1}+1}{q_{n+1}}=1$ in $\R$.

    Moreover, if $A+1$ is $q$-bounded, then $A+1\subseteq^d supp_q(x)$.
\item[(a2)] If $d(A\cap supp(x))=0$, then there exists $A'\subseteq A$ with $d(A\setminus A')=0$ such that $\lim\limits_{n\in A'} \frac{c_{n+1}}{q_{n+1}}=0$ in $\R$.

    Moreover, if $A+1$ is $q$-bounded, then $d((A+1)\cap supp(x))=0$ as well.
\end{itemize}
\item[(b)] If $A$ is $q$-divergent, then $\lim\limits_{n\in B} \frac{c_n}{q_n}=0$ in $\T$ for some $B\subseteq A$ with $d(A\setminus B)=0$.
\end{itemize}
\end{theorem}

\begin{proof}
{\bf Necessity}: Suppose $\overline{d}(supp(x))>0$ and $x\in t^s_{(a_n)}(\T)$. Therefore, there exists $M \subseteq \N$ with $d(M) = 1$ such that
 \begin{equation}\label{M}
 \lim\limits_{n\in M} \{a_{n-1}x\}=0 \mbox{ in $\T$.}
  \end{equation}
  Consider any $A\subseteq\N$ with $\overline{d}(A)>0$. We take $B=M\cap A$. Then $B\subseteq A$ and $d(A\setminus B)=d(A\cap (\N\setminus M))=0$. As $B\subseteq M$, from equation (\ref{M}), we get $\lim\limits_{n\in B} \{a_{n-1}x\}=0$ in $\T$. Consequently, there exists $B\subseteq A$ with $d(A\setminus B)=0$ such that $x\in t_{(a_{B-1})}(\T)$.

 \noindent\textbf{(a)}  Suppose first that $A$ is $q$-bounded. The following two cases can arise:\\

 \textbf{(a1)}  First, suppose $A\subseteq^d supp(x)$. Then $B\subseteq A$ is $q$-bounded and $B\subseteq^d supp(x)$. Since, $x\in t_{(a_{B-1})}(\T)$ and $B$ is $q$-bounded, from Lemma \ref{nec} we conclude that $B\subseteq^d supp_q(x)$ and $\lim\limits_{n\in A'} \{a_{n-1}x\}=1$ in $\R$, where $A'\subseteq B$ with $d(B\setminus A')=0$. \\
     Therefore, from equation (\ref{eq6})
     \begin{eqnarray*}
     1 \ &=& \ \lim\limits_{n\in A'}(\frac{c_n}{q_n}+\frac{\{a_nx\}}{q_n})\ = \ \lim\limits_{n\in A'}(\frac{q_n-1+\{a_nx\}}{q_n}) \\
      &=& \lim\limits_{n\in A'}(1-\frac{1-\{a_nx\}}{q_n})
      \Rightarrow \ \lim\limits_{n\in A'} \frac{1-\{a_nx\}}{q_n}=0.
     \end{eqnarray*}
     Hence, we get
     \begin{equation}\label{B1}
     \lim\limits_{n\in A'} \{a_nx\}=1 \mbox{ (Since, $A'\subseteq B$ is $q$-bounded)}.
     \end{equation}
     Now from the definition of canonical representation (\ref{canrep}), $c_{n+1}\leq q_{n+1}-1$ for all $n\in\N$. Again from equation (\ref{eq6}), we have
     $$
     \{a_nx\}=\frac{c_{n+1}}{q_{n+1}}+\frac{\{a_{n+1}x\}}{q_{n+1}}<\frac{c_{n+1}+1}{q_{n+1}}\leq 1.
     $$
     Hence from equation (\ref{B1}), it follows that
     \begin{equation}\label{B2}
     1 \ = \ \lim\limits_{n\in A'} \{a_nx\}\ \leq\ \lim\limits_{n\in A'}\frac{c_{n+1}+1}{q_{n+1}}\ \leq\ 1 \mbox{ \ ~ \ ~ \ ~ \ i.e. }  \lim\limits_{n\in A'}\frac{c_{n+1}+1}{q_{n+1}}\ = \ 1
     \end{equation}
     Now, $q_{n+1}\geq 2$ for all $n\in\N$. From equation (\ref{B2}), we can observe that $c_{n+1}+1 > 1$ ( i.e. $c_{n+1}\neq0$ ) for almost all $n\in A'$. Which implies $A'+1\subseteq^* supp(x)$. Since, $d(B\setminus A')=0$, we obtain $B+1\subseteq^d supp(x)$.

     As $B\subseteq A$ and $d(A\setminus B)=0$, we must have $A+1\subseteq^d supp(x)$, $A\subseteq^d supp_q(x)$ and $\lim\limits_{n\in A'} \frac{c_{n+1}+1}{q_{n+1}}\ = \ 1$ for some $A'\subseteq A$ where $d(A\setminus A')=0$.
     If $A+1$ is $q$-bounded, proceeding as in the first part of the proof, we get $A+1\subseteq^d supp_q(x)$.\\

 \textbf{(a2)}    Now let $d(A\cap supp(x))=0$. Since $B\subseteq A$, we must have $d(B\cap supp(x))=0$. Then from Lemma \ref{nec}, we can conclude that $\lim\limits_{n\in A'} \{a_{n-1}x\}=0$ in $\R$ for some $A'\subseteq B$ with $d(B\setminus A')=0$. Therefore putting $k=1$ in equation (\ref{eq0}) and equation (\ref{eq5}), we get
     $$
     \lim\limits_{n\in A'} (\frac{c_n}{q_n} + \frac{c_{n+1}}{q_n q_{n+1}} + \frac{\{a_{n+1}x\}}{q_n q_{n+1}}) \ = \ \lim\limits_{n\in A'} \{a_{n-1}x\} \ = \ 0
     $$
     \begin{equation}\label{B3}
     \Rightarrow \ \lim\limits_{n\in A'} \frac{c_{n+1}}{q_n q_{n+1}} \ = \ \lim\limits_{n\in A'} \frac{\{a_{n+1}x\}}{q_n q_{n+1}}\ = \ 0 \mbox{ (Since $c_n, \{a_nx\} \geq 0$ and $q_n>0$ )}.
     \end{equation}
     Now as $A'\subseteq B$ is $q$-bounded, equation (\ref{B3}) implies that $\lim\limits_{n\in A'} \frac{c_{n+1}}{q_{n+1}}=0$ in $\R$, where $A'\subseteq B\subseteq A$ and $d(A\setminus A')=0$.

     Moreover, if $A+1$ is $q$-bounded, then vanishing of the last limit implies that $(A'+1)\cap supp(x)$ is finite. Thus $d((A+1)\cap supp(x))=0$ (Since, $d((A+1)\setminus (A'+1))=d(A\setminus A')=0$).  \\

    \noindent\textbf{(b)} Suppose $A$ is $q$-divergent i.e. $\lim\limits_{n\in A} q_n = \infty$. Then from equation (\ref{eq6}), we get

   $\lim\limits_{n\in B} (\frac{c_n}{q_n}+\frac{\{a_nx\}}{q_n})  =   \lim\limits_{n\in B} \{a_{n-1}x\} =  0 \mbox{ in } \T ~\mbox{for some}~ B\subseteq A ~\mbox{ with}~ d(A\setminus B)=0$\\

    $$\Rightarrow \ \lim\limits_{n\in B} \frac{c_n}{q_n}  =  0 \mbox{ in } \T \ \mbox{ (Since, $\{a_nx\}<1$ and $\lim\limits_{n\in B} q_n = \infty)$}.
    $$

\begin{claim}\label{remark1}
 Before proving the sufficiency of the conditions, we need to reformulate the necessary conditions in a stronger iterated version. For any $A\in [\N]^{\aleph_0}$ and $k\in\N\cup\{0\}$, we define $L_k(A)=\bigcup\limits_{i=0}^{k} (A+i)$. Now putting $k=k+1$ in equation (\ref{eq0}), we obtain
 \begin{equation}\label{sigma}
 \sigma_{n,k+1}=\sigma_{n,k}+\frac{c_{n+k+1}}{q_n q_{n+1} \ldots q_{n+k+1}}.
 \end{equation}
 Therefore, from equation (\ref{eq5}) and equation (\ref{sigma}), it follows that
 \begin{equation}\label{split}
 \{a_{n-1}x\}=\sigma_{n,k+1}+\frac{\{a_{n+k+1}x\}}{q_n q_{n+1} \ldots q_{n+k+1}} \ = \ \sigma_{n,k} + \frac{c_{n+k+1}}{q_n q_{n+1} \ldots q_{n+k+1}}+ \frac{\{a_{n+k+1}x\}}{q_n q_{n+1} \ldots q_{n+k+1}}
 \end{equation}
 \begin{equation}\label{convergence}
 \Rightarrow \ \sigma_{n,k}\leq \ \{a_{n-1}x\} \ < \ \sigma_{n,k} + \frac{c_{n+k+1}}{q_n q_{n+1} \ldots q_{n+k+1}} + \frac{1}{2^{(k+2)}}.
 \end{equation}
  \\
 Let $x\in\T$ has canonical representation (\ref{canrep}) such that (a) and (b) of Theorem \ref{mainth} hold. Let $A\subseteq \N$ be $q$-bounded with $\overline{d}(A)>0$. If $L_k(A)$ is $q$-bounded for some $k\in\N\cup\{0\}$, then the following hold:
 \begin{itemize}
 \item[(i)]  If $A\subseteq^d supp(x)$, then $L_k(A)\subseteq^d supp_q(x)$ and $\lim\limits_{n\in A'+k+1} \frac{c_n+1}{q_n}=1$ in $\R$ for some $A'\subseteq A$ with $d(A\setminus A')=0$. Therefore there exists $n_k\in\N$ such that for all $n\in A'$ with $n\geq n_k$,
     \begin{equation}\label{sigmarec}
     \sigma_{n,k}=1-\frac{1}{q_n q_{n+1} \ldots q_{n+k}} \geq 1-\frac{1}{2^{k+1}}.
     \end{equation}
 Moreover if $A+k+1$ is $q$-divergent, then
 \begin{equation}\label{divergerec}
 \lim\limits_{n\in A+k+1} \frac{c_n}{q_n} \ = \ \lim\limits_{n\in A} \frac{c_{n+k+1}}{q_{n+k+1}} \ = \ 1 \mbox{ in } \R.
 \end{equation}
\item[(ii)]  If $d(A\cap supp(x))=0$, then $d(L_k(A)\cap supp(x))=0$ and $\lim\limits_{n\in A'} \frac{c_{n+k+1}}{q_{n+k+1}}=0$  in $\R$ for some $A'\subseteq A$ and $d(A\setminus A')=0$.
\end{itemize}
\end{claim}
\noindent{\bf Sufficiency}: If $d(supp(x))=0$, then from \cite[Theorem 4.3]{DPK} it readily follows that $x\in t^s_{(a_n)}(\T)$. So let $\overline{d}(supp(x))>0$ and $supp(x)$ satisfy conditions $(a)$ and $(b)$. To show that $x\in t^s_{(a_n)}(\T)$, in view of Lemma \ref{suf1} it is sufficient to check the convergence criterion: for all $A\subseteq\N$ with $\overline{d}(A)>0$, there exists $B'\subseteq A$ such that $\lim\limits_{n\in B'} a_{n-1}x=0$ in $\T$.
Indeed without any loss of generality, we can assume that either $d(A\cap supp(x))=0$ or $A\subseteq^d supp(x)$.\\

\noindent{\bf Case (i)}:  First let $A$ be $q$-bounded.

{\bf Subcase (i$_a$)}:  Let us first assume that $L_k(A)$ is $q$-bounded for all $k\in\N\cup\{0\}$.
    Let $\varepsilon>0$ be given. Choose $k\in\N$ such that $\frac{1}{2^{k+1}}<\varepsilon$.

      $\ast~~$ Let $A\subseteq^d supp(x)$. Then, from (i) of Claim \ref{remark1}, $L_k(A)\subseteq^d supp_q(x)$.  Therefore, there exists $B'\subseteq A$ such that for all $n\in B'$,
        $$
        \sigma_{n,k} =1-\frac{1}{q_n q_{n+1} \ldots q_{n+k}}\geq \ 1-\frac{1}{2^{k+1}}\ > \ 1-\varepsilon
        $$
        $$
        \Rightarrow \  1-\varepsilon<\sigma_{n,k}\ \leq \ \{a_{n-1}x\}\ < 1 \ ~ \  \forall\ n\in B' \mbox{ (From equation (\ref {convergence})).}
        $$

     $\ast~~$   Let $d(A\cap supp(x))=0$. Then, from (ii) of Claim \ref{remark1}, $d(L_k(A)\cap supp(x))=0$ and $\lim\limits_{n\in B} \frac{c_{n+k+1}}{q_{n+k+1}}=0$  in $\R$ for some $B\subseteq A$ with $d(A\setminus B)=0$. So, there exists $B'\subseteq B$ such that $\sigma_{n,k}=0$ and $\frac{c_{n+k+1}}{q_{n+k+1}}<\varepsilon$ for all $n\in B'$. Therefore, from equation (\ref{convergence}), we get
        $$
        \{a_{n-1}x\} < \ \sigma_{n,k} + \frac{c_{n+k+1}}{q_n q_{n+1} \ldots q_{n+k+1}} + \frac{1}{2^{(k+2)}}\ < 2\varepsilon \ \forall \ n\in B'.
        $$

    Thus in both cases, we have $\lim\limits_{n\in B'} \{a_{n-1}x\}=0$ in $\T$ for some $B'\subseteq A$, as required.\\

    {\bf Subcase (i$_b$)}:  We assume that there exists an integer $k\geq0$ such that $A+k+1$ is not $q$-bounded but $A+i$ is $q$-bounded for all $i=0,1,2,\ldots ,k$. If there exists an $A'\subseteq A$ such that $\overline{d}(A')>0$ and $A'+k+1$ is $q$-bounded, then without any loss of generality we can start with $A'$ in place of $A$. If this process does not terminate after finitely many steps then we can conclude that there exists $B\subseteq A$ with $\overline{d}(B)>0$ such that $L_k(B)$ is $q$-bounded for all $k\in\N$. Consequently, we can consider $B$ in place of $A$ and proceed as in Subcase (i$_a$).

     Now let us consider the case when there does not exist any $A'\subseteq A$ such that $\overline{d}(A')>0$ and $A'+k+1$ is $q$-bounded. Therefore from Lemma \ref{suf2}, there exists $B\subseteq A$ with $d(A\setminus B)=0$ such that $B+k+1$ is $q$-divergent i.e. $\lim\limits_{n\in B} q_{n+k+1} =\infty$. Clearly $L_k(B)$ is $q$-bounded. Further more
    \begin{equation}\label{sufeq0}
    \lim\limits_{n\in B} \frac{\{a_{n+k+1}x\}}{q_n q_{n+1} \ldots q_{n+k+1}} \ \leq \ \lim\limits_{n\in B} \frac{1}{q_{n+k+1}}=0.
    \end{equation}
    Therefore, from equation (\ref{split}) and equation (\ref{sufeq0}), we get
    \begin{eqnarray*}\label{sufeq1}
    \lim\limits_{n\in B} \{a_{n-1}x\} \ &=& \ \lim\limits_{n\in B} \sigma_{n,k} +  \lim\limits_{n\in B} \frac{c_{n+k+1}}{q_n q_{n+1} \ldots q_{n+k+1}}+ \lim\limits_{n\in B} \frac{\{a_{n+k+1}x\}}{q_n q_{n+1}\ldots q_{n+k+1}} \\
     &=& \lim\limits_{n\in B} \sigma_{n,k} +  \lim\limits_{n\in B} \frac{c_{n+k+1}}{q_n q_{n+1} \ldots q_{n+k+1}}.
    \end{eqnarray*}

     $\ast~~$ Let $A\subseteq^d supp(x)$. Therefore $B\subseteq^d supp(x)$. Consequently from equation (\ref{sigmarec}) of Claim \ref{remark1} and equation (\ref{sufeq1}), we get
    \begin{eqnarray*}
    \lim\limits_{n\in B'} \{a_{n-1}x\} \ &=& \ \lim\limits_{n\in B'}(1-\frac{1}{q_n q_{n+1}\ldots q_{n+k}}+\frac{c_{n+k+1}}{q_n q_{n+1}\ldots q_{n+k+1}})\\
    &=& \ \lim\limits_{n\in B'}(1+\frac{1}{q_n q_{n+1}\ldots q_{n+k}}\cdot (\frac{c_{n+k+1}}{q_{n+k+1}}-1)) \ = \ 1
    \end{eqnarray*}
    for some $B'\subseteq B$ with $d(B\setminus B')=0$.

      $\ast~~$ Next let $d(A\cap supp(x))=0$. Then there exists $B\subseteq A$ such that $\sigma_{n,k}=0$ for all $n\in B$. Subsequently from (ii) of Claim \ref{remark1} and equation (\ref{sufeq1}), we have
    $$
    \lim\limits_{n\in B'} \{a_{n-1}x\} \ = \ \lim\limits_{n\in B'} \frac{c_{n+k+1}}{q_n q_{n+1}\ldots q_{n+k+1}} \ \leq \ \lim\limits_{n\in B'} \frac{c_{n+k+1}}{q_{n+k+1}}\ =0
    $$
    for some $B'\subseteq B$ with $d(B\setminus B')=0$. Thus in both cases, we again obtain that $\lim\limits_{n\in B'} \{a_{n-1}x\}=0$ in $\T$ for some $B'\subseteq A$.

    \noindent{\bf Case (ii)}:   We assume that $A$ is not $q$-bounded. If there exists $A'\subseteq A$ such that $\overline{d}(A')>0$ and $A'$ is $q$-bounded then we can proceed as in Case (i) and consider $A'$ in place of $A$. So, let us assume that there does not exist any $A'\subseteq A$ such that $\overline{d}(A')>0$ and $A'$ is $q$-bounded. Then from Lemma \ref{suf2}, there exists $B\subseteq A$ with $d(A\setminus B)=0$ such that $B$ is $q$-divergent i.e. $\lim\limits_{n\in B} q_n=\infty$. From hypothesis, we have $\lim\limits_{n\in B'} \frac{c_n}{q_n} = 0$ in $\T$ for some $B'\subseteq B$ with $d(B\setminus B')=0$. Therefore, from equation (\ref{eq6}), we obtain
    $$
    \lim\limits_{n\in B'}\{a_{n-1}x\} \ = \ \lim\limits_{n\in B'}(\frac{c_n}{q_n}+\frac{\{a_nx\}}{q_n}) \ =0 \mbox{ in $\T$  (Since $\lim\limits_{n\in B'} \frac{\{a_nx\}}{q_n} < \lim\limits_{n\in B'} \frac{1}{q_n} \ =0$ ).}
    $$
 Hence in all cases, we can conclude that for any $A\subseteq \N$ with $\overline{d}(A)>0$, there exists $B'\subseteq A$ such that $\lim\limits_{n\in B'} \{a_{n-1}x\}=0$ in $\T$. This shows that $x\in t^s_{(a_n)}(\T)$ i.e. $x$ is a topological $s$-torsion element of $\T$.
\end{proof}

\begin{remark}\label{remark}
Since, for all $n\not\in supp(x)$, we have $c_n=0$, it is sufficient to consider only subsets of $supp(x)$ in item $(b)$ of Theorem \ref{mainth}.
\end{remark}

 Theorem 2.1 obviously provides the complete solution of the open problem Problem 6.10 posed in \cite{DPK}. In the remaining part of the article we follow in the line of investigations of \cite{DI1} which would show that in certain circumstances, one can obtain more simplified  equivalent solutions for  Problem 6.10. Before proceeding further, let us recall the following notion of "splitting" sequences which were considered in \cite{DI1}.

\begin{definition}\cite[Definition 3.10]{DI1}
A sequence $(q_n)$ of natural numbers has the splitting property if there exists a partition $\N= B\cup I$, such that the following statements hold:
\begin{itemize}
\item[(a)]  $B$ and $I$ are either empty or infinite;
\item[(b)]  $I$ is $q$-divergent, in case $I$ is infinite;
\item[(c)]  $B$ is $q$-bounded, in case $B$ is infinite.
\end{itemize}
Here, $B$ and $I$ witness the splitting property for $(q_n)$, where $B$ and $I$ can be uniquely defined up to a finite set.
\end{definition}
\begin{proposition}\label{oldsplit}\cite[Proposition 3.11]{DI1}
A sequence $(q_n)$ has the splitting property if and only if there exists a natural number $M$ such that the set $\{n\in\N: \ q_n\in [M,m]\}$ is finite for every $m>M$.
\end{proposition}

As a natural consequence, we can think of generalizing the idea of a splitting sequence using natural density.
\begin{definition}\label{dsplitdef}
We say that, a sequence $(q_n)$ of natural numbers has the $d$-splitting property if there exists a partition $\N=B\cup D$, such that the following statements hold:
\begin{itemize}
\item[(a)]  $B$ and $D$ are either empty or $\overline{d}(B),\overline{d}(D)>0$.
\item[(b)]  If $\overline{d}(B)>0$, then there exists $B'\subseteq \N$ with $d(B\triangle B')=0$ such that $B'$ is $q$-bounded.
\item[(b)]  If $\overline{d}(D)>0$, then there exists $D'\subseteq \N$ with $d(D\triangle D')=0$ such that $D'$ is $q$-divergent.
\end{itemize}
Here, $B$ and $D$ witness the $d$-splitting property for $(q_n)$, where $B$ and $D$ can be uniquely determined up to a zero density set (i.e. if $B_1\cup D_1$ is another partition of $\N$, witnessing the $d$-splitting property for $(q_n)$, then $B_1=^d B$ and $D_1=^d D$).
\end{definition}

\begin{proposition}\label{newsplit}
A sequence $(q_n)$ has the $d$-splitting property if and only if there exists a natural number $M$ such that $d(\{n\in\N: \ q_n\in [M,m]\})=0$ for every $m>M$.
\end{proposition}
\begin{proof}
We assume that $(q_n)$ has the $d$-splitting property. Now two cases can arise:
\begin{itemize}
\item[*]  At first, we consider $B=\emptyset$. Then there exists a $D'\subseteq\N$ with $d(D')=1$ such that $D'$ is $q$-divergent. Take any $m\in \N$. Since $D'$ is $q$-divergent, there exists an $n_m\in\N$ such that $q_n>m$ for all $n>n_m$ and $n\in D'$. We set $M=1$. Then it is evident that for all $m>M$
    \begin{eqnarray*}\label{dsplitseq}
     \overline{d}(\{n\in\N:  q_n\in[M,m]\}) & \leq&  \overline{d}(\{n\in D':  q_n\in[M,m]\} + \overline{d}(\N\setminus D')\\
     & \leq&  \overline{d}(\{n\in D':  n\leq n_m\})=0.
    \end{eqnarray*}
\item[*]  Let $B\neq\emptyset$. Then we have $\overline{d}(B)>0$ and consequently there exists a $B'\subseteq \N$ with $d(B\triangle B')=0$ such that $B'$ is $q$-bounded. In this case, we set $M=1+ \max\limits_{n\in B'}\{q_n\}$. Therefore, for any $m>M$, we obtain
    \begin{eqnarray*}
    &&\overline{d}(\{n\in\N:  q_n\in[M,m]\})\\
    & \leq& \overline{d}(\{n\in B':  q_n\in[M,m]\}) + \overline{d}(B\setminus B')\\ &+& \overline{d}(\{n\in D':  q_n\in[M,m]\})+\overline{d}(D\setminus D')\\
    &=& \overline{d}(\{n\in D':  q_n\in[M,m]\})=0 \mbox{ (From equation (\ref{dsplitseq})).}
    \end{eqnarray*}
\end{itemize}
 Conversely, let there exists a natural number $M$ such that $d(\{n\in\N:q_n\in[M,m]\})=0$ for all $m>M$. We set $B'=\{n\in\N:q_n\in[1,M-1]\}$ and $D'=\N\setminus B'$.
\begin{itemize}
\item[*]  If $\overline{d}(B')>0$ and $d(D')=0$, then we take $B=\N$ and $D=\emptyset$.
\item[*]  If $d(B')=0$ and $\overline{d}(D')>0$, then we take $D=\N$ and $B=\emptyset$.
\item[*]  If $\overline{d}(B')>0$ and $\overline{d}(D')=>0$, then we take $B=B'$ and $D=D'$.
\end{itemize}
Clearly, $B$ and $D$ witness the $d$-splitting property for the sequence $(q_n)$.
\end{proof}

From Proposition \ref{oldsplit} and Proposition \ref{newsplit}, it is obvious that every splitting sequence is a $d$-splitting sequence. However the converse is not necessarily true, nor it is true that every subset of $\N$ has the $d$-splitting property (an example not having splitting property was given in Example 3.12 \cite{DI1} but one must take into consideration that a non-splitting sequence can still be $d$-splitting).
\begin{example}
Let $A_1=\{n\in\N: n=k^2$ for some $k\in\N\}$, $A_2=\{n\in\N:n=k^2+1$ for some $k\in\N\}\setminus A_1$, $\ldots$ , $A_{i+1}=\{n\in\N:n=k^2+i$ for some $k\in\N\}\setminus \bigcup\limits_{j=1}^{i} A_j$. Take any $n\in\N$. One can find a $k\in\N$ such that $k^2\leq n< (k+1)^2$. So we can write $n=k^2+i$ for some $i\in\N\cup \{0\}$ i.e. $n\in A_i$. Therefore, $\N=\bigcup\limits_{i=1}^{\infty} A_i$ i.e. $(A_i)_{i\in\N}$ forms a partition of $\N$.

For each $m\in\N$, we now define $q_n=m$ for all $n\in A_m$. Clearly, for $m,M\in\N$ and $m>M$, we have $\{n\in\N:q_n\in[M,m]\}=\bigcup\limits_{i=M}^{m} A_i$. Since $d(A_i)=0$ for all $i\in\N$, we get $d(\{n:q_n\in[M,m]\})=0$ for all $m,M\in\N$ and $m>M$. Therefore, from Proposition \ref{newsplit}, $(q_n)$ is a $d$-splitting sequence. But, we can observe that $\{n:q_n\in[M,m]\}$ cannot be finite for any $m,M\in\N$ and $m>M$ (since $A_i$ is infinite for all $i\in\N$). Therefore, from Proposition \ref{oldsplit}, $(q_n)$ is not a splitting sequence.
\end{example}

\begin{example}
 Let us define $q_n=\{i\in\N:n=2^{i-1}(2k-1)$ for some $k\in\N\}$. Let $A_i=\{n\in\N:q_n=i\}$. From the construction it is evident that $d(A_i)=\frac{1}{2^i}$ i.e. $d(A_i)>0$ for all $i\in\N$ and $\N=\bigcup\limits_{i=1}^{\infty} A_i$. Now for any $m,M\in\N$ with $m>M$, observe that $d(\{n\in\N:q_n\in[M,m]\})=d(\bigcup\limits_{i=M}^{m} A_i)>0$. Therefore, from Proposition \ref{newsplit}, $(q_n)$ is not a $d$-splitting sequence.
\end{example}
Motivated by these two examples, we present below equivalent conditions for a sequence to be splitting or $d$-splitting (or in other words, equivalent formulations of Proposition \ref{oldsplit} and Proposition \ref{newsplit}).
\begin{proposition}
Let $(q_n)$ be a sequence of natural numbers. For all $i\in\N$, we define $A_i=\{n:q_n=i\}$. Then
\begin{itemize}
\item[(i)]  $(q_n)$ is a splitting sequence if and only if there does not exist a subsequence $(A_{n_k})_{k\in\N}$ of $(A_n)$ such that $A_{n_k}$ is infinite for all $k\in\N$:
\item[(ii)]  $(q_n)$ is a $d$-splitting sequence if and only if there does not exist a subsequence $(A_{n_k})_{k\in\N}$ of $(A_n)$ such that $\overline{d}(A_{n_k})>0$ for all $k\in\N$.
\end{itemize}
\end{proposition}

  For the next result we will use the following notations. Let $(a_n)$ be an arithmetic sequence and $x\in\T$ with canonical representation (\ref{canrep}). Assume that the sequence of ratios $(q_n)$ has the $d$-splitting property which means that there exists a partition $\N=B\cup D$ such that $(a), (b)$ and $(c)$ of Definition \ref{dsplitdef} hold. We will write $B^S(x)=B\cap supp(x)$, $B^N(x)=B\cap (\N\setminus supp(x))$ and $D^S(x)=D\cap supp(x)$.

 From Remark \ref{remark}, it follows that $D\cap (\N\setminus supp(x))$ does not play any role in Theorem \ref{mainth}. Note that if $B,D \neq \emptyset$, then there exists $B'\subseteq \N$ with $d(B\triangle B')=0$ and $D'\subseteq \N$ with $d(D\triangle D')=0$ such that $B'^S(x),B'^N(x)$ are $q$-bounded while $D'^S(x)$ is $q$-divergent. Our next result is a characterization of a topological $s$-torsion element, when the sequence of ratios $(q_n)$ has the $d$-splitting property.

\begin{theorem}\label{splitcoro}
Let $(a_n)$ be an arithmetic sequence and $x\in\T$ has canonical representation \ref{canrep}. If the sequence of ratios $(q_n)$ has the $d$-splitting property, then $x$ is a topological $s$-torsion element i.e. $x\in t^s_{(a_n)}(\T)$ if and only if the following conditions hold:
\begin{itemize}
\item[(i)]  $B^S(x)+1\subseteq^d supp(x), \ B^S(x)\subseteq^d supp_q(x)$, and if $\overline{d}(B^S(x))>0$ then $\lim\limits_{n\in B_1^S(x)} \frac{c_{n+1}+1}{q_{n+1}} =1$ in $\R$, where $B_1\subseteq B$ with $d(B\setminus B_1)=0$.
\item[(ii)]  If $\overline{d}(B^N(x))>0$, then $\lim\limits_{n\in B_1^N(x)} \frac{c_{n+1}}{q_{n+1}}=0$ in $\R$, where $B_1\subseteq B$ with $d(B\setminus B_1)=0$.
\item[(iii)]  If $\overline{d}(D^S(x))>0$, then $\lim\limits_{n\in D_1^S(x)} \frac{c_n}{q_n}=0$ in $\T$, where $D_1\subseteq D$ with $d(D\setminus D_1)=0$.
\end{itemize}
\end{theorem}
\begin{proof}
{\bf Necessity:}  Let $x\in t^s_{(a_n)}(\T)$. Observe that (a) and (b) of Theorem \ref{mainth} hold. \\
\begin{itemize}
\item[(i)]  If $d(B^S(x)=0$, then there is nothing to prove. So, we consider the case when $\overline{d}(B^Sx)>0$. Now, there exists a $B'\subseteq \N$ with $d(B\triangle B')=0$ such that $B'$ is $q$-bounded. Since $d(B\triangle B')=0$, we get $B'^S(x)\subseteq^d supp(x)$. Therefore, taking $A=B'^S(x)$ in Theorem \ref{mainth} and applying (a1), we get $B'^S(x)+1\subseteq^d supp(x), \ B'^S(x)\subseteq^d supp_q(x)$ and $\lim\limits_{n\in B'^S_1(x)} \frac{c_{n+1}+1}{q_{n+1}} =1$ in $\R$, where $B'_1\subseteq B'$ and $d(B'\setminus B'_1)=0$. \\

    Again, since $d(B\triangle B')=0$, we finally get $B^S(x)+1\subseteq^d supp(x), \ B^S(x)\subseteq^d supp_q(x)$ and $\lim\limits_{n\in B_1^S(x)} \frac{c_{n+1}+1}{q_{n+1}} =1$ in $\R$, where $B_1=(B\cap B'_1)\subseteq B$ with $d(B\setminus B_1)=0$. \\
\item[(ii)]  Let $\overline{d}(B^N(x))>0$. Since $d(B^N(x)\cap supp(x))=0$, applying (a2) of Theorem \ref{mainth} to $A=B'^N(x)$, we get $\lim\limits_{n\in B_1^N(x)} \frac{c_{n+1}}{q_{n+1}}=0$ in $\R$, where $B_1\subseteq B$ and $d(B\setminus B_1)=0$.\\
\item[(iii)]  Let $\overline{d}(D^S(x))>0$. Since there exists $D'\subseteq \N$ with $d(D\triangle D')=0$ such that $D'$ is $q$-divergent, applying (b) of Theorem \ref{mainth} to $A=D'^S(x)$ (As, $d(D\triangle D')=0 \ \Rightarrow\ \overline{d}(D'^S(x))=\overline{d}(D^S(x))>0$), we get $\lim\limits_{n\in D_1^S(x)} \frac{c_n}{q_n}=0$ in $\T$, where $D_1\subseteq D$ and $d(D\setminus D_1)=0$.
\end{itemize}
{\bf Sufficiency:}  Let the conditions hold. It suffices to show that the conditions of Theorem \ref{mainth} hold. If $d(supp(x))=0$ then there is nothing to prove. So let us assume that $\overline{d}(supp(x))>0$. Consider any $A\subseteq N$ with $\overline{d}(A)>0$. \\
\begin{itemize}
\item[(a)]  First suppose that $A$ is $q$-bounded. \\
\begin{itemize}
\item[(a1)]  Let $A\subseteq^d supp(x)$. Since $A$ is $q$-bounded, $A\subseteq^d B$. Therefore, $A\subseteq^d B^S(x)$ and we get $\overline{d}(B^S(x)>0$. By (i), we have $B^S(x)+1\subseteq^d supp(x), \ B^S(x)\subseteq^d supp_q(x)$ and $\lim\limits_{n\in B_1^S(x)} \frac{c_{n+1}+1}{q_{n+1}} =1$ in $\R$, where $B_1\subseteq B$ and $d(B\setminus B_1)=0$. Again, since $A\subseteq^d B^S(x)$, we get $A+1\subseteq^d supp(x), \ A\subseteq^d supp_q(x)$ and $\lim\limits_{n\in A'} \frac{c_{n+1}+1}{q_{n+1}}=1$ in $\R$, where $A'=A\cap B_1 \subseteq A$ and $d(A\setminus A')=0$.
\item[(a2)]  Now, let $d(A\cap supp(x))=0$. Since $A$ is $q$-bounded, $A\subseteq^d B$. Therefore, $A\subseteq^d B^N(x)$ and we get $\overline{d}(B^N(x))>0$. By (ii), we have $\lim\limits_{n\in B_1^N(x)} \frac{c_{n+1}}{q_{n+1}}=0$ in $\R$, where $B_1\subseteq B$ and $d(B\setminus B_1)=0$. Now, taking $A'=A\cap B_1$, we get
    $\lim\limits_{n\in A'} \frac{c_{n+1}}{q_{n+1}}=0$ in $\R$, where $A'\subseteq A$ and $d(A\setminus A')=0$ ( Since, $d(A\setminus A')=d(A\setminus B_1)=d(A\setminus B)=0$). \\
\end{itemize}
\item[(b)]  Let us now assume that $A$ is $q$-divergent. Then we have $A\subseteq^d D$. From Remark \ref{remark}, without any loss of generality we can assume that $A\subseteq supp(x)$. Therefore, $A\subseteq^d D^S(x)$ and we get $\overline{d}(D^S(x))>0$. By (iii), we have $\lim\limits_{n\in D_1^S(x)} \frac{c_n}{q_n}=0$ in $\T$, where $D_1\subseteq D$ and $d(D\setminus D_1)=0$. Now, taking $A'=A\cap D_1$, we get $\lim\limits_{n\in A'} \frac{c_n}{q_n}=0$ in $\T$, where $A'\subseteq A$ and $d(A\setminus A')=0$ ( Since, $d(A\setminus A')=d(A\setminus D_1)=d(A\setminus D)=0$). Therefore, from theorem \ref{mainth}, we can conclude that $x\in t^s_{(a_n)}(\T)$.
\end{itemize}
\end{proof}


In particular, one can obtain simpler characterizations of topological $s$-torsion elements when $supp(x)$ is either $q$-bounded or $q$-divergent for the given arithmetic sequence.
\begin{corollary}\label{supboucoro}
If $supp(x)$ is $q$-bounded, then $x\in t^s_{(a_n)}(\T)$ if and only if the following statements hold:
\begin{itemize}
\item[(i)]  $d((supp(x)+1)\setminus supp(x))=0$, and  \item[(ii)]   $ \ d(supp(x)\setminus supp_q(x))=0$.
\end{itemize}
\end{corollary}
\begin{proof}
Let $x\in t^s_{(a_n)}(\T)$. If $d(supp(x))=0$ then there is nothing to prove. So, assume that $\overline{d}(supp(x))>0$. Since $supp(x)$ is $q$-bounded and $\overline{d}(supp(x))>0$, we set $A=supp(x)$. Therefore from item (a1) of Theorem \ref{mainth}, we get $supp(x)+1\subseteq^d supp(x)$ and $supp(x)\subseteq^d supp_q(x)$. Thus, we have (i) $d((supp(x)+1)\setminus supp(x))=0$, and  (ii)   $ \ d(supp(x)\setminus supp_q(x))=0$.

In order to prove the sufficiency of the conditions, if possible, suppose that there is a $x\in t^s_{(a_n)}(\T)$ for which (i) does not hold i.e $\overline{d}((supp(x)+1)\setminus supp(x))>0$. Since, $\overline{d}(supp(x)+1)=\overline{d}(supp(x))$, we must have $\overline{d}(supp(x))>0$. Now, taking $A=supp(x)$ and applying item (a1) of Theorem \ref{mainth}, we get $A+1\subseteq^d supp(x)$ i.e. $d(supp(x)+1\setminus supp(x))=0$ \ ~ \ $-$ which is a contradiction. Therefore (i) holds true.

Now, let us consider that (ii) does not hold i.e. $\overline{d}(supp(x)\setminus supp_q(x))>0$ but $x\in t^s_{(a_n)}(\T)$. Set $A=supp(x)\setminus supp_q(x)$ and $q=1+\max\limits_{n\in supp(x)} \{q_n\}$. Consequently, from equation (\ref{eq6}), we obtain
$$
\frac{1}{q}< \lim\limits_{n\in A} \frac{c_n}{q_n}\leq \lim\limits_{n\in A} \{a_{n-1}x\} <\lim\limits_{n\in A} \frac{c_n+1}{q_n}\leq \lim\limits_{n\in A} \frac{q_n-1}{q_n}=1- \lim\limits_{n\in A}\frac{1}{q_n}<1-\frac{1}{q}
$$
$$
\Rightarrow \ \lim\limits_{n\in A} \{a_{n-1}x\} \neq 0 \mbox{ in $\T$ for some $\overline{d}(A)>0$}
$$
$-$ Which is a contradiction. Therefore (ii) holds true.
\end{proof}

\begin{corollary}\label{supdivcoro}
If $supp(x)$ is $q$-divergent, then $x\in t^s_{(a_n)}(\T)$ if and only if the following statements hold:
\begin{itemize}
\item[(i)]  $\lim\limits_{n\in D'} \frac{c_n}{q_n}=0$ in $\T$ for some $D'\subseteq supp(x)$ with $d(supp(x)\setminus D')=0$; and
\item[(ii)]  For every $D\subseteq^d supp(x)$ such that $D-1$ is $q$-bounded, $\lim\limits_{n\in D'} \frac{c_n}{q_n}=0$ in $\R$, where $D'\subseteq D$ and $d(D\setminus D')=0$.
\end{itemize}
\end{corollary}
\begin{proof}
First, let $x\in t^s_{(a_n)}(\T)$. If $d(supp(x))=0$, then there is nothing to prove. So, let us assume that $\overline{d}(supp(x))>0$. Now, taking $A=supp(x)$ and applying item (b) of Theorem \ref{mainth}, we can conclude that (i) holds true. Next let us suppose that $A=D-1$ is $q$-bounded for some $D\subseteq supp(x)$. If $d(A)=d(D)=0$, then there is nothing to prove. Therefore, we can assume $\overline{d}(A)>0$. Since, $supp(x)$ is $q$-divergent, we have $d(A\cap supp(x))=0$. Now, applying item (a2) of Theorem \ref{mainth}, we have $\lim\limits_{n\in A'}\frac{c_{n+1}}{q_{n+1}}=0$ in $\R$ for some $A'\subseteq A$ with $d(A\setminus A')=0$. Putting $D'=A'+1$, we get $\lim\limits_{n\in D'} \frac{c_n}{q_n}=0$ in $\R$, where $D'\subseteq D$ and $d(D\setminus D')=0$.

Conversely, let us assume that the conditions hold. To prove that $x\in t^s_{(a_n)}(\T)$, we need to show (a) and (b) of Theorem \ref{mainth} hold. Since (b) follows from (i), it is sufficient to show only (a). If $d(supp(x))=0$, then $x\in t^s_{(a_n)}(\T)$. So, assume that $\overline{d}(supp(x))>0$. Now, take any $A\subseteq N$ with $\overline{d}(A)>0$. If $A$ is $q$-bounded, then $d(supp(x)\cap A)=0$. Therefore, we need to prove only (a2). If $d((A+1)\cap supp(x))=0$, then taking $A'+1=(A+1)\setminus supp(x)$, we get $\lim\limits_{n\in A'} \frac{c_{n+1}}{q_{n+1}}=0$ in $\R$, where $A'\subseteq A$ with $d(A\setminus A')=0$. Now considering the situation when $\overline{d}((A+1)\cap supp(x))>0$, taking $D=(A+1)\cap supp(x)$ and applying (ii), we get $\lim\limits_{n\in D'}\frac{c_n}{q_n}=0$ in $\R$ for some $D'\subseteq D$ with $d(D\setminus D')=0$. Thus, putting $A'=D'-1$ in a similar manner, we obtain that $\lim\limits_{n\in A'} \frac{c_{n+1}}{q_{n+1}}=0$ in $\R$ for some $A'\subseteq A$ with $d(A\setminus A')=0$. Therefore, (a2) holds and we finally have $x\in t^s_{(a_n)}(\T)$.
\end{proof}

Following observations follow from our main results, giving certain particular cases of an element of $\T$ being or not being a topological $s$-torsion element.\\

$\bullet$ If $supp(x)$ is $q$-divergent and $\lim\limits_{n\in A} \frac{c_n}{q_n} =0$ in $\R$ for some $A\subseteq supp(x)$ with $d(supp(x)\setminus A)=0$, then $x$ is a topological $s$-torsion element of $\T$.

$\bullet$ Suppose $x\in\T$ has canonical representation \ref{canrep} with $q$-divergent support. If $d(supp(x)\setminus \{n\in supp(x): (c_n)$ is bounded$\})=0$, then $x$ is a topological $s$-torsion element of $\T$.

$\bullet$
Suppose $A$ is $q$-divergent and $d(A)=1$. Then $x$ is a topological $s$-torsion element of $\T$ if and only if $\lim\limits_{n\in D'} \frac{c_n}{q_n}=0$ in $\T$ for some $D'\subseteq supp(x)$ with $d(supp(x)\setminus D')=0$.

$\bullet$
Let $(a_n)$ be an arithmetic sequence and $x\in\T$ be such that
\begin{itemize}
\item[(i)] $supp(x)=\bigcup\limits_{n=1}^{\infty}[p_n,r_n]$, $p_n,r_n\in\N$, $p_n\leq r_n < p_{n+1}$ for all $n\in \N$;
\item[(ii)]  there exist $l\in\N$ such that for all $n\in\N$, $|r_n-p_n|\leq l$ and $|p_{n+1}-r_n|\leq l$;
\item[(iii)]  $supp(x)$ is $q$-bounded.
\end{itemize}
 Then $x$ is not a topological $s$-torsion element of $\T$.

$\bullet$
Let $(a_n)$ be an arithmetic sequence and $x\in\T$ be such that
\begin{itemize}
\item[(i)]  $\overline{d}(supp(x))>0$ and $supp(x)$ is $q$-divergent;
\item[(ii)]  for all $n\in supp(x)$, $\frac{c_n}{q_n}\in [r_1,r_2]$, where $0<r_1,r_2<1$.
\end{itemize}
 Then $x$ is not a topological $s$-torsion element of $\T$.\\

\noindent\textbf{Acknowledgement:}  The second author is also thankful to the CSIR for granting Junior Research Fellowship during the tenure of which this work was done.\\

\end{document}